\documentclass[11pt,leqno]{article}
\usepackage[margin=1in]{geometry} 
\usepackage{amssymb,amsfonts,amsmath,bbm,mathrsfs,stmaryrd,mathtools}
\usepackage{xcolor}
\usepackage{url}

\usepackage{graphicx}

\usepackage{accents}

\usepackage{extarrows}

\usepackage[shortlabels]{enumitem}
\usepackage{tensor}

\usepackage{xr}
\usepackage[T1]{fontenc}
\usepackage[utf8]{inputenc}

\usepackage[colorlinks,
linkcolor=black!75!red,
citecolor=blue,
pdftitle={},
pdfproducer={pdfLaTeX},
pdfpagemode=None,
bookmarksopen=true,
bookmarksnumbered=true,
backref=page]{hyperref}

\usepackage{tikz}
\usetikzlibrary{arrows,calc,decorations.pathreplacing,decorations.markings,decorations.shapes,intersections,shapes.geometric,through,fit,shapes.symbols,positioning,decorations.pathmorphing}

\makeatletter
\newlength\zig@L
\newlength\zig@La
\newlength\zig@Lb

\newcommand{\xzigrightarrow}[2][]{%
  \mathrel{%
    \settowidth{\zig@La}{$\scriptstyle #2$}%
    \settowidth{\zig@Lb}{$\scriptstyle #1$}%
    \zig@L=\zig@La\relax
    \ifdim\zig@Lb>\zig@L \zig@L=\zig@Lb\fi
    \advance\zig@L by 2.2em\relax
    \tikz[baseline=-0.65ex]{%
      \draw[->,
            line cap=round,
            decorate,
            decoration={zigzag,segment length=4pt,amplitude=1.1pt}]%
        (0,0) -- (\zig@L,0)
        node[midway,above=2pt] {$\scriptstyle #2$}%
        \if\relax\detokenize{#1}\relax\else
          node[midway,below=2pt] {$\scriptstyle #1$}%
        \fi
      ;
    }%
  }%
}
\makeatother

\makeatletter
\newcommand{\squigjoin}{1mu} 

\def\sqleft@{\sim}                    
\def\sqmid@{\sim\mkern-\squigjoin}    

\def\rightsquigarrowfill@{%
  \arrowfill@{\sqleft@}{\sqmid@}{\mkern-4mu\succ}%
}

\newcommand{\xrightsquigarrow}[2][]{%
  \ext@arrow 0359\rightsquigarrowfill@{#1}{#2}%
}
\makeatother

\makeatletter
\newcommand*\circled[1]{\tikz[baseline=(char.base)]{
    \node[shape=circle, draw, inner sep=0pt, 
    minimum height={\f@size},] (char) {\vphantom{WAH1g}#1};}}
\makeatother

\makeatletter
\DeclareRobustCommand\widecheck[1]{{\mathpalette\@widecheck{#1}}}
\def\@widecheck#1#2{%
    \setbox\z@\hbox{\m@th$#1#2$}%
    \setbox\tw@\hbox{\m@th$#1%
       \widehat{%
          \vrule\@width\z@\@height\ht\z@
          \vrule\@height\z@\@width\wd\z@}$}%
    \dp\tw@-\ht\z@
    \@tempdima\ht\z@ \advance\@tempdima2\ht\tw@ \divide\@tempdima\thr@@
    \setbox\tw@\hbox{%
       \raise\@tempdima\hbox{\scalebox{1}[-1]{\lower\@tempdima\box
\tw@}}}%
    {\ooalign{\box\tw@ \cr \box\z@}}}
\makeatother

\usepackage{braket}

\usepackage[amsmath,thmmarks,hyperref]{ntheorem}
\usepackage{cleveref}

\newcommand\nthalias[1]{\AddToHook{env/#1/begin}{\crefalias{lemma}{#1}}}

\nthalias{definition}
\nthalias{example}
\nthalias{examples}
\nthalias{remark}
\nthalias{remarks}
\nthalias{convention}
\nthalias{notation}
\nthalias{construction}
\nthalias{sketch}
\nthalias{theoremN}
\nthalias{propositionN}
\nthalias{corollaryN}
\nthalias{lemma}
\nthalias{proposition}
\nthalias{corollary}
\nthalias{theorem}
\nthalias{conjecture}
\nthalias{question}
\nthalias{assumption}

\creflabelformat{enumi}{#2#1#3}

\crefname{section}{Section}{Sections}
\crefformat{section}{#2Section~#1#3} 
\Crefformat{section}{#2Section~#1#3} 

\crefname{subsection}{\S}{\S\S}
\AtBeginDocument{%
  \crefformat{subsection}{#2\S#1#3}%
  \Crefformat{subsection}{#2\S#1#3}%
}

\crefname{subsubsection}{\S}{\S\S}
\AtBeginDocument{%
  \crefformat{subsubsection}{#2\S#1#3}%
  \Crefformat{subsubsection}{#2\S#1#3}%
}

\theoremstyle{plain}

\newtheorem{lemma}{Lemma}[section]
\newtheorem{proposition}[lemma]{Proposition}
\newtheorem{corollary}[lemma]{Corollary}
\newtheorem{theorem}[lemma]{Theorem}

\theoremstyle{plain}
\theoremnumbering{Alph}

\theoremstyle{plain}
\theorembodyfont{\upshape}
\theoremsymbol{\ensuremath{\blacklozenge}}

\newtheorem{definition}[lemma]{Definition}
\newtheorem{example}[lemma]{Example}

\newtheorem{remark}[lemma]{Remark}
\newtheorem{remarks}[lemma]{Remarks}

\newtheorem{notation}[lemma]{Notation}

\crefname{definition}{definition}{definitions}
\crefformat{definition}{#2definition~#1#3} 
\Crefformat{definition}{#2Definition~#1#3} 

\crefname{ex}{example}{examples}
\crefformat{example}{#2example~#1#3} 
\Crefformat{example}{#2Example~#1#3} 

\crefname{exs}{example}{examples}
\crefformat{examples}{#2example~#1#3} 
\Crefformat{examples}{#2Example~#1#3} 

\crefname{remark}{remark}{remarks}
\crefformat{remark}{#2remark~#1#3} 
\Crefformat{remark}{#2Remark~#1#3} 

\crefname{remarks}{remark}{remarks}
\crefformat{remarks}{#2remark~#1#3} 
\Crefformat{remarks}{#2Remark~#1#3} 

\crefname{convention}{convention}{conventions}
\crefformat{convention}{#2convention~#1#3} 
\Crefformat{convention}{#2Convention~#1#3} 

\crefname{notation}{notation}{notations}
\crefformat{notation}{#2notation~#1#3} 
\Crefformat{notation}{#2Notation~#1#3} 

\crefname{table}{table}{tables}
\crefformat{table}{#2table~#1#3} 
\Crefformat{table}{#2Table~#1#3}

\crefname{lemma}{lemma}{lemmas}
\crefformat{lemma}{#2lemma~#1#3} 
\Crefformat{lemma}{#2Lemma~#1#3} 

\crefname{proposition}{proposition}{propositions}
\crefformat{proposition}{#2proposition~#1#3} 
\Crefformat{proposition}{#2Proposition~#1#3} 

\crefname{propositionN}{proposition}{propositions}
\crefformat{propositionN}{#2proposition~#1#3} 
\Crefformat{propositionN}{#2Proposition~#1#3} 

\crefname{corollary}{corollary}{corollaries}
\crefformat{corollary}{#2corollary~#1#3} 
\Crefformat{corollary}{#2Corollary~#1#3} 

\crefname{corollaryN}{corollary}{corollaries}
\crefformat{corollaryN}{#2corollary~#1#3} 
\Crefformat{corollaryN}{#2Corollary~#1#3} 

\crefname{theorem}{theorem}{theorems}
\crefformat{theorem}{#2theorem~#1#3} 
\Crefformat{theorem}{#2Theorem~#1#3} 

\crefname{theoremN}{theorem}{theorems}
\crefformat{theoremN}{#2theorem~#1#3} 
\Crefformat{theoremN}{#2Theorem~#1#3} 

\crefname{enumi}{}{}
\crefformat{enumi}{#2#1#3}
\Crefformat{enumi}{#2#1#3}

\crefname{assumption}{assumption}{Assumptions}
\crefformat{assumption}{#2assumption~#1#3} 
\Crefformat{assumption}{#2Assumption~#1#3} 

\crefname{construction}{construction}{Constructions}
\crefformat{construction}{#2construction~#1#3} 
\Crefformat{construction}{#2Construction~#1#3} 

\crefname{sketch}{sketch}{Sketches}
\crefformat{sketch}{#2sketch~#1#3} 
\Crefformat{sketch}{#2Sketch~#1#3} 

\crefname{question}{question}{Questions}
\crefformat{question}{#2question~#1#3} 
\Crefformat{question}{#2Question~#1#3} 

\crefname{equation}{}{}
\crefformat{equation}{(#2#1#3)} 
\Crefformat{equation}{(#2#1#3)}

\numberwithin{equation}{section}

\theoremstyle{nonumberplain}
\theoremsymbol{\ensuremath{\blacksquare}}

\newtheorem{proof}{Proof}
\newcommand\pf[1]{\newtheorem{#1}{Proof of \Cref{#1}}}

\newcommand\bG{{\mathbb G}}

\newcommand\bK{{\mathbb K}}

\newcommand\bZ{{\mathbb Z}}

\newcommand\cC{{\mathcal C}}

\newcommand\cM{{\mathcal M}}

\newcommand\cO{{\mathcal O}}

\newcommand\cY{{\mathcal Y}}
\newcommand\cZ{{\mathcal Z}}

\newcommand\wt{\widetilde}

\DeclareMathOperator{\id}{id}

\DeclareMathOperator{\End}{\mathrm{End}}

\newcommand{\cat}[1]{\textsc{#1}}

\newcommand{\qedhere}{\mbox{}\hfill\ensuremath{\blacksquare}}

\newcommand{\comment}[1]{}

\newcommand{\xrightarrowdbl}[2][]{%
  \xrightarrow[#1]{#2}\mathrel{\mkern-14mu}\rightarrow
}

\title{Quantum Hopf rigidity in symmetric monoidal categories}
\author{Alexandru Chirvasitu}

\begin{document}

\date{}

\newcommand{\Addresses}{{
  \bigskip
  \footnotesize

  \textsc{Department of Mathematics, University at Buffalo}
  \par\nopagebreak
  \textsc{Buffalo, NY 14260-2900, USA}  
  \par\nopagebreak
  \textit{E-mail address}: \texttt{achirvas@buffalo.edu}

}}

\maketitle

\begin{abstract}
  We prove that a coaction of a Hopf algebra on another preserving the latter's associative-algebra structure and comultiplication automatically also preserves the counit, antipode, and tensor-interchange map and in fact factors through the coaction dual to an action of an affine group scheme by Hopf-algebra automorphisms. This generalizes and unifies a number of results to the effect that quantum groups have only classical symmetries, due to Kasprzak et al., Budzi\'nski-Kasprzak and Brannan et al.

  The stated automatic classicality follows from the fact that the antipode, counit and tensorand-interchange of a Hopf algebra internal to any symmetric monoidal category are contained in the (non-symmetric) monoidal subcategory generated by the associative-algebra structure and the comultiplication, assuming a weak form of arrow-retraction invariance.
\end{abstract}


\vspace{.5cm}

\noindent{MSC 2020: 16T05; 18M80; 16T15; 18M05; 18D15; 18A25

  
}


\section*{Introduction}

The motivation for the present note is a cluster of results bearing a strong family resemblance, all to the effect that Hopf-like structures exhibit what one might call \emph{quantum rigidity}: coactions thereon by Hopf-algebra-type gadgets exhibit some (perhaps initially unexpected) commutativity, so that said coactions are in fact effectively actions by ordinary groups (discrete, algebraic, etc., depending on context). A brief sampling that will help convey the flavor might run as follows. 
\begin{enumerate}[(a),wide]
\item\label{item:intro.cqg.qaut.cls} Consider a \emph{compact quantum group} $\bG$ in the sense of \cite[Definition 1.1.1]{NeTu13} acting on a \emph{finite} quantum group $\bK$, so as to preserve the latter's structure: the representative-function algebra $\cO(\bK)$ is both a comodule-$*$-algebra and a comodule-$*$-coalgebra over $\cO(\bG)$. The action $\bG\circlearrowright \bK$ must then factor \cite[Corollary 3.3]{MR3305979} through that of a plain, classical group acting on $\bK$. Pithily: the \emph{quantum automorphism group} $\mathrm{qAut}(\bK)$ (introduced in \cite[Theorem 2.11]{MR3283728}) is in fact classical.

\item\label{item:intro.cqg.qaut.cls.bis} \cite[Corollary 7.2]{2608.14858v1} recovers this by pictorial techniques familiar from quantum-mechanics literature employing graphical Frobenius-algebra manipulation (e.g. \cite[\S\S 5, 6]{MR2503596}). 
  
\item\label{item:intro.qfam.qaut.cls} \cite[Theorem 3.4]{MR3576681} generalizes the result in the context of \emph{quantum families of quantum-group morphisms}: for Hopf algebras $H_i$, associative-algebra morphisms $H_1\to A\otimes H_2$ intertwining comultiplications in the appropriate sense factor through $A'\otimes H_2$ for commutative $A'\le A$.
\end{enumerate}

What ultimately drives the phenomenon is the possibility of recovering the tensorand-interchange map (henceforth \emph{flip}) $x\otimes y\mapsto y\otimes x$ on the tensor square of a Hopf algebra by operations available in the language of monoidal categories (arbitrary, \emph{not} symmetric) performed only on the associative (co)algebra structures. This is in evidence in essentially this form in \cite[Proposition 7.1]{2608.14858v1}, with the caveat that a somewhat richer structure is assumed there than will be present below (the monoidal categories in question are \emph{rigid}, \emph{dagger} \cite[\S 2]{zbMATH06236130}, etc.).

Concretely, one target result lying upstream of (so specializing to) the preceding instances is the following principle. 

\begin{theorem}\label{th:hopf.rec.flip}
  Let $H$ be a Hopf algebra internal to a symmetric monoidal category $(\cC,\otimes,\mathbf{1},\tau)$.

  The flip $\tau_{H,H}$, antipode $S=S_H$ and (co)unit $\varepsilon$, $\eta$ belong to the monoidal (non-symmetric) subcategory $\cC'\subseteq \cC$ generated by either $(\mu,\Delta, \eta)$ or $(\mu,\Delta, \varepsilon)$ and closed under arrow retraction in the following weak sense:
  \begin{equation}\label{eq:weak.rtrct}
    \forall\left(\text{arrows }\pi,s\in \cC,\ \pi s=\id\ \wedge\ s\pi\in \cC'\right)    
    \left(
      s\in \cC'
      \iff
      \pi\in \cC'
    \right)
  \end{equation}
\end{theorem}

This is enough to recover the classicality results surveyed in \Cref{item:intro.cqg.qaut.cls}-\Cref{item:intro.qfam.qaut.cls} and substantially more: the quantum automorphism group of \emph{any} Hopf algebra, under any sensible construction one might put on the phrase, is classical.  

\begin{theorem}\label{th:cls.act.arb.halg}
  Let $(K,\mu,\eta,\Delta,\varepsilon,S)$ be a Hopf algebra and $B$ a bialgebra.  
  \begin{enumerate}[(1),wide]
  \item\label{item:th:cls.act.arb.halg:bialg} A coaction $K\xrightarrow{\rho} K\otimes B$ preserving both the unital associative-algebra structure and the comultiplication factors through a coaction by a commutative sub-bialgebra $B'\le B$ preserving the entire Hopf structure $(K,\mu,\eta,\Delta,\varepsilon,S)$.

  \item\label{item:th:cls.act.arb.halg:halg} If $B$ is Hopf, the coaction $\rho$ in \Cref{item:th:cls.act.arb.halg:bialg} factors through a commutative Hopf subalgebra $B'\le B$, so is the coaction dual to a Hopf-structure-preserving $\bG$-action on $K$ for an affine group scheme $\bG$.

  \item\label{item:th:cls.act.arb.halg:halg.act} An action $B\otimes K \to K$ preserving both the unital associative-algebra structure and the comultiplication factors through an action by a cocommutative quotient bialgebra $B\xrightarrowdbl{} B'$ preserving the entire Hopf structure $(K,\mu,\eta,\Delta,\varepsilon,S)$.
  \end{enumerate}
\end{theorem}

\subsection*{Acknowledgments}

I am grateful for illuminating exchanges with A. Agore, M. Brannan and D. Gromada. 


\section{Antipode, counit and flip recovery via (co)multiplication}\label{se:rec.ant.cnit.flp}

We assume some familiarity with the language of (symmetric) monoidal categories and closure thereof, as covered, say, in \cite[\S 6.1]{brcx_hndbk-2}. This and other relevant category-theoretic machinery will be sourced more precisely as needed. 

The focus is on Hopf algebras $(H,\mu,\eta,\Delta,\varepsilon)$ internal to such categories (the maps are the multiplication, unit, comultiplication and counit in this order); these are discussed (indeed, in broader braided contexts) in \cite[\S 5.2]{MR2793022}, for instance.

\pf{th:hopf.rec.flip}
\begin{th:hopf.rec.flip}
  Arrow reversal renders one variant superfluous, so we assume the morphisms $(\mu,\Delta,\eta)$ available.

  \begin{enumerate}[(I)]
  \item \textbf{: $\tau_{H,H}$, given the antipode $S$.} \cite[p. 775]{MR2793022} (see also \cite[Theorem]{MR1656075}) notes that $\tau_{H,H}$ is expressible as 
    \begin{equation*}
      \tau_{H,H}
      =
      (\mu\otimes \mu)
      \circ
      \left(S\otimes \Delta\mu\otimes S\right)
      \circ
      \left(\Delta\otimes \Delta\right),
    \end{equation*}
    so it suffices to recover the antipode $S\in \cC(H,H)$.

  \item \textbf{: $S$, given the counit $\varepsilon$.} \cite[Lemma 5.1(a)]{MR2793022}:
    \begin{equation*}
      S
      =
      (\varepsilon\otimes \id)
      \circ
      \bigg((\id\otimes \mu)(\Delta\otimes\id)\bigg)^{-1}
      \circ
      (\id\otimes\eta):
    \end{equation*}
    \Cref{eq:weak.rtrct} entails closure under arrow inversion, so that operation is indeed available. 
    
  \item \textbf{: Counit.} As the morphism
    \begin{equation*}
      \eta\varepsilon
      =
      \mu
      \circ
      \bigg((\id\otimes \mu)(\Delta\otimes\id)\bigg)^{-1}
      \circ
      (\id\otimes\eta)
    \end{equation*}
    is available, so too is $\varepsilon$ by \Cref{eq:weak.rtrct} applied to $\pi:=\varepsilon$ and $s:=\eta$.  \qedhere
  \end{enumerate}
\end{th:hopf.rec.flip}

A number of consequences follow: \cite[Theorem 3.4]{MR3576681}, in turn generalizing \cite[Corollary 3.3]{MR3305979}, is what \Cref{cor:qtm.fam.maps}'s first bullet point specializes to when internalizing to plain vector spaces; the second and third items, on the other hand, recover \cite[Theorems 3.5 and 3.6]{MR3576681} respectively.

\begin{notation}\label{not:mrg.as}
  For morphisms
  \begin{equation*}
    x_i
    \xmapsto{\quad f_i\quad}
    a\otimes y_i
    ,\quad
    \begin{gathered}
      1\le i\le n\\
      x_i,y_i,a\in \text{symmetric monoidal }(\cC,\otimes,\mathbf{1},\tau)\\
      a\text{ internal associative algebra}
    \end{gathered}
  \end{equation*}
  write
  \begin{equation*}
    \begin{tikzpicture}[>=stealth,auto,baseline=(current  bounding  box.center)]
      \path[anchor=base] 
      (0,0) node (l) {$x_1\otimes\cdots\otimes x_n$}
      +(-2,1.5) node (u) {$\left(x_1\otimes a\right)\otimes\cdots\otimes \left(x_n\otimes a\right)$}
      +(6,1.5) node (ur) {$x_1\otimes\cdots\otimes x_n\otimes a^{\otimes n}$}
      +(4,0) node (r) {$x_1\otimes\cdots\otimes x_n\otimes a$}
      ;
      \draw[->] (l) to[bend left=6] node[pos=.5,auto] {$\scriptstyle \otimes_i f_i$} (u);
      \draw[->] (u) to[bend left=6] node[pos=.5,auto] {$\scriptstyle \sigma$} (ur);
      \draw[->] (ur) to[bend left=6] node[pos=.3,auto] {$\scriptstyle \id\otimes \mu^{(n-1)}$} (r);
      \draw[->] (l) to[bend right=6] node[pos=.5,auto,swap] {$\scriptstyle \boxtimes_i f_i$} (r);      
    \end{tikzpicture}
  \end{equation*}
  with $a^{\otimes n}\xrightarrow{\mu^{(n-1)}}a$ denoting iterated multiplication for $n\ge 1$ and $a$'s unit for $n=0$, and $\sigma$ the \emph{shuffle} \cite[p. 248]{swe} moving all $a$ tensorands to the right by applying instances of the flip maps $\tau_{\bullet,\bullet}$, preserving the order of both the $x_i$s on the one hand and the $a$ tensorands on the other.
  
  Note also the obvious analogue for morphisms $\bullet\otimes c\to \bullet$, $c$ being an internal (coassociative, counital) coalgebra.
\end{notation}

Internalizing \cite[Definition 3.2]{MR3576681} to arbitrary symmetric monoidal categories:  

\begin{definition}\label{def:qfam}
  Let $H_i$, $i=1,2$ be two bialgebras in a symmetric monoidal category $(\cC,\otimes,\mathbf{1},\tau)$. For a (unital) associative algebra $A\in \cC$ an \emph{$A$-based quantum family of morphisms $H_1\to H_2$} is a unital-algebra morphism $H_1\to A\otimes H_2$ that intertwines the two comultiplications.
\end{definition}

A simple general remark will help check \Cref{eq:weak.rtrct} in \Cref{cor:qtm.fam.maps} below. Consider the diagram
\begin{equation*}
  \begin{tikzpicture}[>=stealth,auto,baseline=(current  bounding  box.center)]
    \path[anchor=base] 
    (0,0) node (l) {$\cY_d$}
    +(2,.5) node (u) {$\cY$}
    +(6,0) node (r) {$\cZ$}
    +(3,-.8) node (alpha) {$\scriptstyle \alpha\Downarrow$}
    ;
    \draw[right hook->] (l) to[bend left=6] node[pos=.5,auto] {$\scriptstyle \iota $} (u);
    \draw[->] (u) to[bend left=20] node[pos=.5,auto] {$\scriptstyle F_1$} (r);
    \draw[->] (u) to[bend left=10] node[pos=.5,auto,swap] {$\scriptstyle F_2$} (r);
    \draw[->] (l) to[bend right=15] node[pos=.5,auto] {$\scriptstyle F_1\iota$} (r);
    \draw[->] (l) to[bend right=35] node[pos=.5,auto,swap] {$\scriptstyle F_2\iota$} (r);
  \end{tikzpicture}
\end{equation*}
of functors and natural transformations, with the ``$d$'' subscript denoting the \emph{discretization} of a category: its identity morphisms only. This provides
\begin{itemize}[wide]
\item functors
  \begin{equation*}
    \overrightarrow{\cY}
    \ni
    f\xmapsto{\quad\wt{F}_1\quad}
    F_1(\mathrm{dom}~f)
    \in \cZ
    ,\quad
    \overrightarrow{\cY}
    \ni
    f\xmapsto{\quad\wt{F}_2\quad}
    F_2(\mathrm{codom}~f)
    \in \cZ
  \end{equation*}
  on the \emph{arrow category} of $\cY$;

\item two natural transformations
  \begin{equation*}
    \begin{gathered}
      f
      \xmapsto{\quad\alpha_2\quad}
      \left(
        F_1(\mathrm{dom}~f)
        \xrightarrow{\quad\alpha_{\mathrm{codom}~f}\circ F_1f\quad}
        F_2(\mathrm{codom}~f)
      \right)\\
      f
      \xmapsto{\quad\alpha_1\quad}
      \left(
        F_1(\mathrm{dom}~f)
        \xrightarrow{\quad F_2f\circ \alpha_{\mathrm{dom}~f}\quad}
        F_2(\mathrm{codom}~f)
      \right)
    \end{gathered}    
  \end{equation*}
  between those functors' restrictions to the discretization $\left(\overrightarrow{\cY}\right)_d$;
  
\item and hence the associated \emph{equifier} \cite[Lemma 2.76]{ar} $\mathrm{Eq}\left(\alpha_{i=1,2}\right)\subseteq\left(\overrightarrow{\cY}\right)_d$: the full subcategory where the two natural transformations agree. 
\end{itemize}
The observation alluded to preceding this background is as follows.

\begin{lemma}\label{le:auto.pres.inv.splt}
  The subcategory $\mathrm{Eq}(\alpha_i)\subseteq \left(\overrightarrow{\cY}\right)_d$ satisfies \Cref{eq:weak.rtrct}. 
\end{lemma}
\begin{proof}
  $\alpha$ simply provides a family of morphisms $F_1 y\to F_2 y$ for objects $y\in \cY$, and $\mathrm{Eq}:=\mathrm{Eq}(\alpha_i)$ is by construction the collection of arrows on which $\alpha$ is natural. Being closed under composition (and containing all identity morphisms), $\mathrm{Eq}$ constitutes the largest subcategory of $\cY$ on which $\alpha$ operates as a natural transformation between the $F_i$. We argue one of the two parallel cases, namely the implication $s\Rightarrow\pi$ (with $y\xrightarrow{\pi}y'$, say).
  
  Consider, to that end, the diagram
  \begin{equation*}
    \begin{tikzpicture}[>=stealth,auto,baseline=(current bounding box.center)]

      \def\depthax{.15}
      \def\depthay{-.45}
      \def\depthbx{-.15}
      \def\depthby{.45}

      \path[anchor=base]
      (0,0)      node (l)  {$F_1y$}
      +(5,.5)    node (u)  {$F_1y$}
      +(4,-.5)   node (d)  {$F_2y$}
      +(8,0)     node (r)  {$F_2y$}
      +(1,-2)  node (lp) {$F_1y$}
      +(6,-1.5)   node (up) {$F_1 y'$}
      +(5,-2.5)  node (dp) {$F_2y$}
      +(9,-2)  node (rp) {$F_2 y'$}
      ;

      \path[name path=u-up]
      (u) .. controls
      +(\depthax,\depthay)
      and +(\depthbx,\depthby)
      .. (up);

      \path[name path=d-r]
      (d) to[bend right=6] (r);

      \path[name path=lp-up]
      (lp) to[bend left=6] (up);

      \path[name path=d-dp]
      (d) .. controls
      +(\depthax,\depthay)
      and +(\depthbx,\depthby)
      .. (dp);

      \path[name intersections={of=u-up and d-r,   by=xupper}];
      \path[name intersections={of=lp-up and d-dp, by=xlower}];

      \draw[<-]
      (u) .. controls
      +(\depthax,\depthay)
      and +(\depthbx,\depthby)
      .. node[pos=.55,right,xshift=5pt,yshift=-4pt] {$\scriptstyle F_1 s$} (up);

      \draw[->]
      (lp) to[bend left=6]
      node[pos=.3,right,xshift=-4pt,yshift=8pt] {$\scriptstyle F_1\pi$} (up);

      \draw[<-]
      (l) .. controls
      +(\depthax,\depthay)
      and +(\depthbx,\depthby)
      .. node[pos=.55,left,xshift=1pt,yshift=1pt] {$\scriptstyle \id$} (lp);

      \draw[<-]
      (r) .. controls
      +(\depthax,\depthay)
      and +(\depthbx,\depthby)
      .. node[pos=.55,right,xshift=1pt,yshift=1pt] {$\scriptstyle  F_2 s$} (rp);

      \draw[->] (l) to[bend left=6]
      node[pos=.5,auto] {$\scriptstyle F_1 (s\pi)$} (u);

      \draw[->] (u) to[bend left=6]
      node[pos=.5,auto] {$\scriptstyle \alpha$} (r);

      \draw[->] (l) to[bend right=6]
      node[pos=.5,auto,swap] {$\scriptstyle \alpha$} (d);

      \fill[white] (xupper) circle[radius=1.6pt];

      \draw[->] (d) to[bend right=6]
      node[pos=.7,auto,swap] {$\scriptstyle  F_2(s\pi)$} (r);

      \draw[->] (up) to[bend left=6]
      node[pos=.5,auto] {$\scriptstyle $} (rp);

      \draw[->] (lp) to[bend right=6]
      node[pos=.5,auto,swap] {$\scriptstyle \alpha$} (dp);

      \draw[->] (dp) to[bend right=6]
      node[pos=.5,auto,swap] {$\scriptstyle  F_2\pi$} (rp);

      \fill[white] (xlower) circle[radius=1.6pt];

      \draw[<-]
      (d) .. controls
      +(\depthax,\depthay)
      and +(\depthbx,\depthby)
      .. node[pos=.55,right,xshift=3pt,yshift=-6pt] {$\scriptstyle \id$} (dp);

    \end{tikzpicture}
  \end{equation*}
  The top and vertical faces commute by hypothesis, so the bottom must also: $F_2$ preserves the left invertibility of the morphism $s$, hence its monic character.
\end{proof}

\begin{corollary}\label{cor:qtm.fam.maps}
  Let $H_{1,2}$ be two Hopf algebras and $A$ a unital associative algebra. A quantum family of morphisms $H_1\xrightarrow{\alpha}A\otimes H_2$ automatically intertwines
  \begin{itemize}[wide]
  \item flips, in the sense that
    \begin{equation*}
      \begin{aligned}
        \rho
        \alpha^{\otimes 2}
        \tau_{H_1,H_1}
        &=
          \left(\id_A\otimes \tau_{H_2,H_2}\right)
          \rho
          \alpha^{\otimes 2}\\
        \rho
        &:=
          \left(\mu_A\otimes \id_{H_2^{\otimes 2}}\right)
          \left(\id_A\otimes \tau_{H_2,A}\otimes \id_{H_2}\right);
      \end{aligned}
    \end{equation*}
    
  \item counits;
    
  \item and antipodes. 
  \end{itemize}  
\end{corollary}
\begin{proof}
  Consider the symmetric monoidal category $\cC$ freely generated by a Hopf object $(H,\mu,\eta,\Delta,\varepsilon)\in \cC$ \cite[Example 1.1]{MR4861254}, equipped with its resulting symmetric monoidal functors $\cC\xrightarrow{F_i}\cat{Vect}$ respectively sending $H$ to $H_i$, $i=1,2$.

  The aim is to apply \Cref{th:hopf.rec.flip} to the subcategory $\cC'\subseteq \cC$ consisting of those morphisms $f\in \cC\left(H^{\otimes m},H^{\otimes n}\right)$ for which $\alpha$ intertwines $F_i f$ in the sense that
  \begin{equation*}
    \begin{tikzpicture}[>=stealth,auto,baseline=(current  bounding  box.center)]
      \path[anchor=base] 
      (0,0) node (l) {$H_1^{\otimes m}$}
      +(3,.5) node (u) {$H_1^{\otimes n}$}
      +(3,-.5) node (d) {$A\otimes H_2^{\otimes m}$}
      +(6,0) node (r) {$A\otimes H_2^{\otimes n}$}
      ;
      \draw[->] (l) to[bend left=6] node[pos=.5,auto] {$\scriptstyle F_1 f$} (u);
      \draw[->] (u) to[bend left=6] node[pos=.5,auto] {$\scriptstyle \alpha^{\boxtimes n}$} (r);
      \draw[->] (l) to[bend right=6] node[pos=.5,auto,swap] {$\scriptstyle \alpha^{\boxtimes m}$} (d);
      \draw[->] (d) to[bend right=6] node[pos=.5,auto,swap] {$\scriptstyle \id_A\otimes F_2 f$} (r);
    \end{tikzpicture}
  \end{equation*}
  commutes ($\alpha^{\boxtimes \bullet}$ as in \Cref{not:mrg.as}). All this requires is verifying \Cref{eq:weak.rtrct}, which follows from \Cref{le:auto.pres.inv.splt} with
  \begin{itemize}[wide]
  \item our $F_1$ and $A\otimes F_2$ respectively in place of \Cref{le:auto.pres.inv.splt}'s $F_{1,2}$;    
  \item and $\alpha$ extended to all objects $H^{\otimes n}$ by means of \Cref{not:mrg.as}'s $\alpha^{\boxtimes\bullet}$ construct. 
  \end{itemize}
\end{proof}

\Cref{th:cls.act.arb.halg}\Cref{item:th:cls.act.arb.halg:bialg} follows:

\begin{proposition}\label{pr:com.subbialg}
  Let $K$ and $B$ be a Hopf algebra and bialgebra respectively. 
  \begin{enumerate}[(1),wide]
  \item\label{item:pr:com.subbialg:coact} A coaction $K\xrightarrow{\rho} K\otimes B$ preserving both the unital associative-algebra structure and the comultiplication factors through a Hopf-structure-preserving coaction by a commutative sub-bialgebra $B'\subseteq B$.

  \item\label{item:pr:com.subbialg:act} An action $B\otimes K\to K$ preserving both the unital associative-algebra structure and the comultiplication factors through a Hopf-structure-preserving action by a cocommutative quotient bialgebra $B\xrightarrowdbl{}B'$.
  \end{enumerate}
\end{proposition}
\begin{proof}
  \begin{enumerate}[label={},wide]
  \item\textbf{\Cref{item:pr:com.subbialg:coact}} Per \Cref{cor:qtm.fam.maps}, the \emph{coefficient coalgebra}
    \begin{equation*}
      C:=
      \left\{\left(\id\otimes f\right)\rho(k)\ :\ k\in K,\ f\in B^*\right\}
      \lhook\joinrel\xrightarrow[\quad\text{inclusion}\quad]{\quad\iota\quad}
      B
    \end{equation*}
    of the coaction is commutative under $B$'s multiplication:
    \begin{equation}\label{eq:c2b.comm}
      \begin{tikzpicture}[>=stealth,auto,baseline=(current  bounding  box.center)]
        \path[anchor=base] 
        (0,0) node (l) {$C\otimes C$}
        +(2,.5) node (u) {$B\otimes B$}
        +(4,0) node (r) {$B$}
        ;
        \draw[->] (l) to[bend left=6] node[pos=.5,auto] {$\scriptstyle \iota^{\otimes 2}$} (u);
        \draw[->] (u) to[bend left=6] node[pos=.5,auto] {$\scriptstyle \mu_B\text{ or }\mu_B\circ\tau_{B,B}$} (r);
        \draw[->] (l) to[bend right=6] node[pos=.5,auto,swap] {$\scriptstyle $} (r);
      \end{tikzpicture}
    \end{equation}
    commutes for a single common bottom arrow, regardless of the right-hand top-arrow choice. The coaction thus factors through one by the \emph{free commutative bialgebra} \cite[Theorem 5.3.6]{rad} $C\to B_{c}(C)$ on $C$ along a bialgebra morphism $B_{c}(C)\to B$ factoring $C\le B$, and we can take for $B'\le B$ the image of that morphism. 

  \item\textbf{\Cref{item:pr:com.subbialg:act}} Much of the preceding argument simply dualizes categorically: the image $B\xrightarrowdbl{}A$ of the morphism $B\to \End(K)$ is \emph{co}commutative under $\Delta_B$ in the sense that the diagram dual to \Cref{eq:c2b.comm} commutes, so the action must factor via $B\to B_{coc}(A)\to A$ through the \emph{cofree cocommutative bialgebra} \cite[Theorem 5.5.3]{rad} on $A$; the sought-after quotient is the image of $B\to B_{coc}(A)$. 
  \end{enumerate}
\end{proof}

\pf{th:cls.act.arb.halg}
\begin{th:cls.act.arb.halg}
  The outstanding matter is the Hopf case in \Cref{item:th:cls.act.arb.halg:halg}. That will follow from the earlier bialgebra version provided we argue that the \emph{free Hopf algebra} $B\to H(B)$ (\cite[Theorem 2.6.3]{par_qg-ncg}, \cite[Theorem 7 2.]{porst_formal-2}) on a commutative bialgebra $B$ is again commutative. This is \cite[Proposition 3.14]{MR5016690}; we provide an alternative proof for completeness. 

  $B$ being a union \cite[Theorem 2.2.3]{rad} of finite-dimensional subcoalgebras which in turn are quotients of \emph{matrix} coalgebras $M_n(\Bbbk)^*$, it will suffice to argue the point for the \emph{free commutative bialgebra} $B:=S\left(M_n(\Bbbk)^*\right)$ of \cite[Definition 5.3.7]{rad}.

  Indeed: a morphism $B\to H$ into a Hopf algebra will make the defining $n$-dimensional $B$-comodule $V$ into an $H$-comodule, with the flip on $V^{\otimes 2}$ being an $H$-comodule morphism. This gives a 1-dimensional comodule $\bigwedge^n V\in \cM^H$ (top exterior power), which must be invertible ($H$ being Hopf). This suffices to conclude: $B\to H$ factors through $B\to \cO(GL_n)$, the regular-function algebra on the affine group scheme $GL_n$ operating on $V$. 
\end{th:cls.act.arb.halg}

\Cref{th:cls.act.arb.halg} recovers the classical nature of the quantum automorphism group of a finite quantum group (so finite-dimensional $K$). There, all Hopf algebras in sight are \emph{CQG} \cite[Definition 2.2]{dk_cqg}): cosemisimple complex $*$-algebras with additional positivity constraints imposed on the unique \cite[Exercise 5.5.9]{dnr} unital (left and right) \emph{Haar integral}. That quantum rigidity result is proven in \cite[Corollary 3.3]{MR3305979} and again via diagrammatic calculus in \cite[Proposition 7.1 and Corollary 7.2]{2608.14858v1}.

\begin{remarks}\label{res:coprod.coc.coalg.0coc}
  \begin{enumerate}[(1),wide]
  \item In amplifying \Cref{th:cls.act.arb.halg}'s item \Cref{item:th:cls.act.arb.halg:bialg} to \Cref{item:th:cls.act.arb.halg:halg} we appealed to the automatic commutativity of the \emph{Hopf reflection} of a commutative bialgebra. \cite[p. 686]{MR5016690} first paragraph asks whether the dual version holds: is the \emph{cofree Hopf algebra} (one exists: \cite[diagram (9)]{porst_formal-2}) $H_{\ell}(B)\to B$ on a cocommutative bialgebra again cocommutative?

    \cite[Remark 26]{porst_formal-2} claims an affirmative answer to that question, but the argument appears to be flawed for the same reasons as those pointed out in \cite[Remark 3.15]{MR5016690}: the concrete construction of $H_{\ell}(B)$ involves \cite[proof of Theorem 3.1]{zbMATH05696924} products in the category of bialgebras (equivalently, coalgebras), and those need not (will not, generally: \Cref{ex:coprod.coc.0coc}) preserve cocommutativity.

  \item In reference to the preceding point, note nevertheless that cocommutative bialgebras do have \emph{largest} Hopf subalgebras:    
    
    The coproduct $H$ of all Hopf subalgebras $\le B$ for cocommutative $B$ in the category of Hopf algebras (equivalently, (bi)algebras) will again be cocommutative. The bialgebra quotient $H\xrightarrowdbl{} H'\le B$ is in fact a Hopf quotient by \cite[Theorem 1(vi)]{nic} (or the earlier \cite[Lemma 4.2]{MR374182} it cites). $H'$ is the sought-after unique inclusion-maximal Hopf subalgebra of $B$.
  \end{enumerate}
\end{remarks}

\begin{example}\label{ex:coprod.coc.0coc}
  Consider any family $\left(\bG_i\right)_{i\in I}$ of finite groups whatsoever. I claim that the coalgebra product (equivalently, Hopf-algebra or bialgebra product: the relevant forgetful functors are right adjoints \cite[Theorem 10 2.]{porst_formal-2}) of the complex function algebras $\cO(\bG_i)$ cannot be cocommutative.

  Indeed, the free product (i.e. group coproduct) $\bG:=\coprod_i \bG_i$ is \emph{residually finite} \cite[Theorem 2]{MR144949} and non-abelian under the two-non-trivial-$\bG_i$ assumption, so the \emph{Bohr compactification} $\bG\to \wt{\bG}$ (i.e. the coproduct of the $\bG_i$ in the category of compact groups) is also non-abelian. The Hopf algebra $\cO\left(\wt{\bG}\right)$ of \emph{representative functions} \cite[Definition 3.3]{hm5} is thus not cocommutative; as its morphisms $\cO\left(\wt{\bG}\right)\xrightarrowdbl{} \cO(\bG_i)$ cannot factor through any proper Hopf quotient, the coalgebra product $\prod \cO(\bG_i)$ cannot be cocommutative either.
\end{example}

While \Cref{th:int.inv} and \Cref{cor:coact.on.fd.hopf} are both consequences of \Cref{th:cls.act.arb.halg}, we provide alternative proofs for the purpose of gleaning further automatic invariance phenomena under Hopf coactions perhaps amenable to internalizing in broader classes of symmetric monoidal categories.

Recall \cite[\S 5.0]{swe} the notion of \emph{integral} on a Hopf algebra $H$ over a field $\Bbbk$ (specialized here to that context to fix ideas; the concept internalizes \cite[Definition 3.1]{MR1759389} to braided monoidal categories):
\begin{equation*}
  \begin{gathered}
    \int^{\ell}=\int^{\ell}_{H^*}
    :=
    \tensor*[^H]{\cM}{}(H,\Bbbk)    
    =
    \left\{f\in H^*\ :\ \forall\left(g\in H^*\right)\left(gf=g(1)f\right)\right\}\\
    \int^{r}=\int^{r}_{H^*}
    :=
    \cM^H(H,\Bbbk)
    =
    \left\{f\in H^*\ :\ \forall\left(g\in H^*\right)\left(fg=g(1)f\right)\right\}
  \end{gathered}  
\end{equation*}
with $\cM^H$ denoting right $H$-comodules and similarly for its left-handed version. The two spaces $\int^{\ell,r}$ are simultaneously 0- or 1-dimensional \cite[Theorem 10.9.5(a)]{rad}, and in the latter case there is a uniquely determined $g\in H$ (the Hopf algebra's \emph{distinguished grouplike} \cite[Proposition 10.9.8]{rad}) with
\begin{equation*}
  \forall\left(f\in H^*\right)
  \forall\left(\psi\in \int^{r}\right)
  \left(f\psi = f(g)\psi\right).
\end{equation*}

\begin{theorem}\label{th:int.inv}
  A Hopf coaction $K\xrightarrow{\rho}K\otimes H$ on a Hopf algebra preserving the unital-algebra structure and the comultiplication preserves the left and right integral spaces $\int^{\ell,r}_{K^*}$ and fixes the distinguished grouplike $g\in K$. 
\end{theorem}
\begin{proof}
  \begin{enumerate}[(I),wide]
  \item\textbf{: Integral-space invariance.} The $H$-invariance of $K$'s entire Hopf structure $(\mu,\eta,\Delta,\varepsilon,S)$ is automatic by \Cref{th:hopf.rec.flip} (and \Cref{cor:qtm.fam.maps}, for the purpose of verifying the former's hypothesis), along with the flip map $\tau_{K,K}$. The following, then, are also $H$-invariant:
    \begin{itemize}[wide]
    \item the associative-algebra structure on $K^*$ dual to $(\Delta,\varepsilon)$;
    \item hence also the \emph{finite dual} $K^{\circ}$, for it is \cite[Proposition 6.0.3]{swe} the pullback
      \begin{equation*}
        \begin{tikzpicture}[>=stealth,auto,baseline=(current  bounding  box.center)]
          \path[anchor=base] 
          (0,0) node (l) {$K^{\circ}$}
          +(2,.5) node (u) {$K^*$}
          +(2,-.5) node (d) {$K^*\otimes K^*$}
          +(5,0) node (r) {$\left(K\otimes K\right)^*$}
          ;
          \draw[right hook->] (l) to[bend left=6] node[pos=.5,auto] {$\scriptstyle $} (u);
          \draw[->] (u) to[bend left=6] node[pos=.5,auto] {$\scriptstyle \mu^*$} (r);
          \draw[->] (l) to[bend right=6] node[pos=.5,auto,swap] {$\scriptstyle $} (d);
          \draw[right hook->] (d) to[bend right=6] node[pos=.5,auto,swap] {$\scriptstyle $} (r);
        \end{tikzpicture}
      \end{equation*}

    \item with it, the right \emph{$K$-Hopf-module} structure $K^{\circ}\in \cM^K_K$ of \cite[Theorem 5.2.1]{dnr}.
    \end{itemize}
    The left-handed conclusion follows from \cite[Lemma 5.2.2]{dnr}'s realization $\int^{\ell}_{K^*}=\left(K^{\circ}\right)^K$; the argument's mirror image runs parallel.
    
  \item\textbf{: Distinguished grouplike.} Writing $[-,-]$ for \emph{internal homs} in closed monoidal categories \cite[\S 1.5]{kly}, note that the line $\Bbbk g\le K\le K^{**}$ can be recovered as the pullback
    \begin{equation*}
      \begin{tikzpicture}[>=stealth,auto,baseline=(current  bounding  box.center)]
        \path[anchor=base] 
        (0,0) node (l) {$\Bbbk g$}
        +(2,.5) node (u) {$\mathbf{1}$}
        +(2,-.5) node (d) {$\left[K^*,\mathbf{1}\right]$}
        +(5,0) node (r) {$\left[K^*,\left[\int^r,\int^r\right]\right]$}
        ;
        \draw[->] (l) to[bend left=6] node[pos=.5,auto] {$\scriptstyle $} (u);
        \draw[->] (u) to[bend left=6] node[pos=.5,auto] {$\scriptstyle $} (r);
        \draw[->] (l) to[bend right=6] node[pos=.5,auto,swap] {$\scriptstyle $} (d);
        \draw[->] (d) to[bend right=6] node[pos=.5,auto,swap] {$\scriptstyle \left[K^*,\mathrm{can}\right]$} (r);
      \end{tikzpicture}
    \end{equation*}
    where $\mathbf{1}\xmapsto{\mathrm{can}}[-,-]$ denotes the canonical morphism adjoint to the identity and the top right-hand arrow is derived from the multiplication $K^*\otimes \int^r\to \int^r$. The diagram being interpretable $H$-equivariantly, $\Bbbk g\le K$ is $H$-invariant and hence the $H$-coaction thereon is \emph{$\gamma$-graded} for some grouplike $\gamma\in H$:
    \begin{equation*}
      \Bbbk g\ni x
      \xmapsto{\quad}
      x\otimes \gamma.
    \end{equation*}
    That coaction also preserving $K$'s comultiplication by assumption, $\Delta g=g\otimes g$ implies that $\gamma$ is idempotent and hence 1. 
  \end{enumerate}
\end{proof}

\begin{remark}\label{re:int.not.fixed}
  If \Cref{th:int.inv}'s $K$ is cosemisimple, or equivalently \cite[Theorem 10.8.2]{rad} has a left (or right) integral not vanishing on $1$, the $H$-coaction will fix the unique unital (left \emph{and} right) integral. In general, the theorem's conclusion cannot be strengthened in this fashion even when $\dim K<\infty$.

  The Hopf algebra $K:=H_{2,-1}$ of \cite[\S 7.3]{rad} in field characteristic $\ne 2$, defined as
  \begin{equation*}
    H_{2,-1}
    :=
    \Braket{g,x\ |\ g^2=1,\ x^2=0,\ gxg^{-1}=-x}
    ,\quad
    \begin{gathered}
      \Delta g=g\otimes g\\
      \Delta x = x\otimes g+1\otimes x
    \end{gathered}
  \end{equation*}
  admits a $\bZ$-grading (i.e. an action by the multiplicative algebraic group $\bG_m$ \cite[\S 1.1]{wat_gpsch}) according $g$ degree 0 and $x$ degree 1. It is not difficult to see that the non-zero left integrals of $H_{2,-1}$ will not annihilate $x$, so will not be invariant under the action/grading.
\end{remark}

The following alternative approach to \Cref{th:cls.act.arb.halg}'s finite-dimensional-$K$ case relies on \Cref{th:int.inv} only, so will be available in contexts where the latter survives but \Cref{th:cls.act.arb.halg} perhaps does not.

\begin{corollary}\label{cor:coact.on.fd.hopf}
  A Hopf coaction $K\to K\otimes H$ on a finite-dimensional Hopf algebra preserving both the unital-algebra structure and the comultiplication factors through a commutative-Hopf-algebra coaction.
\end{corollary}
\begin{proof}
  The universal Hopf algebra coacting on $K$ in the stated fashion is that \emph{Tannaka-reconstructed} by means of \cite[Theorems 2.1.12 and 2.4.2]{schau_tann} from the inclusion functor into $\cat{Vect}_f$ (finite-dimensional vector spaces) of the \emph{rigid} \cite[Definition 2.10.11]{egno} monoidal subcategory of $\cat{Vect}_f$ generated by $(\mu,\eta,\Delta)$, hence also containing $\varepsilon$, $S$ and the flip $\tau_{K,K}$ by \Cref{th:hopf.rec.flip}: the closure property will be automatic, by \Cref{le:auto.pres.inv.splt}. 

  Recall \cite[Theorem 2.1.3]{mont} that finite-dimensional Hopf algebras $K$ are \emph{Frobenius}: $K\overset{\psi}{\cong} K^*$ as right $K$-modules, with $\psi$ implemented by the non-degenerate bilinear form
  \begin{equation*}
    \begin{gathered}
      K\otimes K
      \xrightarrow{\quad\mu\quad}
      K
      \xrightarrow{\quad\int\quad}
      \mathbf{1}
    \end{gathered}    
  \end{equation*}
  for a left integral $\int\in \int^{\ell}_{K^*}$. The commutative diagram
  \begin{equation*}
    \begin{tikzpicture}[>=stealth,auto,baseline=(current  bounding  box.center)]
      \path[anchor=base] 
      (0,0) node (l) {$K^{\otimes 2}$}
      +(2,.5) node (u) {$K^{\otimes 2}$}
      +(2,-.5) node (d) {$K\otimes K^*$}
      +(5,0) node (r) {$K^*\otimes K$}
      ;
      \draw[->] (l) to[bend left=6] node[pos=.5,auto] {$\scriptstyle \tau_{K,K}$} (u);
      \draw[->] (u) to[bend left=6] node[pos=.5,auto] {$\scriptstyle \psi\otimes \id$} (r);
      \draw[->] (l) to[bend right=6] node[pos=.5,auto,swap] {$\scriptstyle \id\otimes \psi$} (d);
      \draw[->] (d) to[bend right=6] node[pos=.5,auto,swap] {$\scriptstyle \tau_{K,K^*}$} (r);
    \end{tikzpicture}
  \end{equation*}
  exhibits $\tau_{K,K^*}$ as a composition of morphisms with the following properties:
  \begin{itemize}[wide] 
  \item $\tau_{K,K}$ is $H$-equivariant as already noted;
  \item by \Cref{th:int.inv}, $\psi\otimes \id$ generates a one-dimensional $H$-subcomodule of $\cat{Vect}\left(K^{\otimes 2},K^*\otimes K\right)\in \cM^H$, so that the coaction is expressible as $\psi\mapsto \psi\otimes g$ for a grouplike $g\in H$;

  \item and by the same token, $\id\otimes \psi^{-1}$ is $g^{-1}$-graded. 
  \end{itemize}
  All in all, then, $\tau_{K,K^*}$ is $H$-equivariant. 

  As the monoidal category generated by the Hopf structure on $K$ already contains the flip $\tau_{K,K}$, it will contain $\tau_{K,K^*}$ also and thus be a \emph{symmetric} rigid monoidal subcategory of $\cat{Vect}$. This forces commutativity on the Tannaka-reconstructed coacting Hopf algebra, hence the statement's factorization.
\end{proof}


\addcontentsline{toc}{section}{References}

\def\polhk#1{\setbox0=\hbox{#1}{\ooalign{\hidewidth
  \lower1.5ex\hbox{`}\hidewidth\crcr\unhbox0}}}


\Addresses

\end{document}